\documentclass{amsart}
\usepackage{amsmath,amssymb,amsthm,fullpage,verbatim}
\usepackage{tikz}
\newtheorem{theorem}{Theorem}
\newtheorem{lemma}{Lemma}

\begin{document}
\title{More on Periodicity and Duality associated with Jordan partitions}
\author{Michael~J.~J.~Barry}
\address{15 River Street Unit 205\\
Boston, MA 02108}
\email{mbarry@allegheny.edu}


\begin{abstract}
Let $J_r$ denote a full $r \times r$ Jordan block matrix with eigenvalue $1$ over a field $F$ of characteristic $p$.  For positive integers $r$ and $s$ with $r \leq s$, the Jordan canonical form of the $r s \times r s$ matrix $J_{r} \otimes J_{s}$ has the form $J_{\lambda_1} \oplus J_{\lambda_2} \oplus \dots \oplus J_{\lambda_{r}}$ where $\lambda_1 \geq \lambda_2 \geq \dots \geq \lambda_{r}>0$.  This decomposition determines a partition $\lambda(r,s,p)=(\lambda_1,\lambda_2,\dots, \lambda_{r})$ of $r s$, known as the \textbf{Jordan partition}, but the values of the parts depend on $r$, $s$, and $p$.  Write
\[(\lambda_1,\lambda_2,\dots, \lambda_{r})=(\overbrace{\mu_1,\dots,\mu_1}^{m_1},\overbrace{\mu_2,\dots,\mu_2}^{m_2},\dots, \overbrace{\mu_k,\dots,\mu_k}^{m_k})
=(m_1 \cdot \mu_1, \dots,m_k \cdot \mu_k),\]
where $\mu_1>\mu_2>\dots>\mu_k>0$, and denote the composition $(m_1,\dots,m_k)$ of $r$ by $c(r,s,p)$.
A recent result of Glasby, Praeger, and Xia in \cite{GPX} implies that if $r \leq p^\beta$, $c(r,s,p)$ is periodic in the second variable $s$ with period length $p^\beta$ and exhibits a reflection property within that period.  We determine the least period length and we exhibit new partial subperiodic and partial subreflective behavior.
\end{abstract}

\maketitle

\section{Introduction}\label{Intro}
The partition $\lambda(r,s,p)$ is also related to the modular representation of a cyclic group $G$ of order $p^\alpha$ where $s \leq p^\alpha$.  The group ring $F G$, where $F$ is a field of characteristic $p$, has exactly $p^\alpha$ pairwise nonisomorphic indecomposable $F G$-modules.  If $V_1$,\dots,$V_{p^\alpha}$ denote $p^\alpha$ pairwise nonisomorphic indecomposable $F G$-modules, then they can be labeled such that $\dim V_i=i$, and then if $r \leq s \leq p^\alpha$, $V_r \otimes V_s=V_{\lambda_1} \oplus \dots \oplus V_{\lambda_r}$ where $\lambda(r,s,p)=(\lambda_1,\lambda_2,\dots, \lambda_{r})$ .

If the Jordan partition
\[\lambda(r,s,p)=(\overbrace{\mu_1,\dots,\mu_1}^{m_1},\overbrace{\mu_2,\dots,\mu_2}^{m_2},\dots, \overbrace{\mu_k,\dots,\mu_k}^{m_k})
=(m_1 \cdot \mu_1, \dots,m_k \cdot \mu_k),\]
where $\mu_1>\mu_2>\dots>\mu_k>0$,
we denote the composition $(m_1,\dots,m_k)$ of $r$ by $c(r,s,p)$.  Note that $\lambda(r,s,p)$ is recoverable from $c(r,s,p)$ and $s$ by~\cite[Proposition 2]{B2015} and~\cite[Theorem 5]{GPX2}.

We record two properties of $c(r,s,p)$ which follow from a result of Glasby, Praeger, and Xia~\cite[Theorem 4]{GPX}.
\begin{theorem}[GPX]\label{GPX}
Suppose that $r \leq p^n$.  Then
\begin{enumerate}
\item $c(r,s,p)=c(r,s+p^n,p)$ for every integer $s \geq r$, \qquad and
\item $c(r,p^n+i,p)=rev(c(r,p^n+(p^n-i),p))$ for every $i \in [0,p^n]$ where $rev((m_1,\dots,m_k))=(m_k,\dots,m_1)$.
\end{enumerate}
\end{theorem}

The first result of Theorem~GPX implies that $c(r,s,p)$ is periodic in the variable $s$ with a period of length $p^n$.  But could a period of length a proper divisor of $p^n$ exist?  Our first result rules this out.
 \begin{theorem}\label{Theorem1}
If $p$ is a prime number and $p^{n-1}<r \leq p^n$, then $c(r,s,p)$ is periodic in $s$ with smallest period of length $p^n$.
\end{theorem}

In our next result, we record partial subperiodic behavior.
\begin{theorem}\label{Theorem2}
Suppose that $p^{n-1}<r \leq \frac{p-1}{2} p^{n-1}$.  Then 
\begin{enumerate}
\item $c(r,s,p)=c(r,s+p^{n-1},p)$ for every integer $s$ satisfying $r \leq s \leq (p-\lceil r/p^{n-1} \rceil )p^{n-1}$, \qquad and
\item $c(r,(p-\lceil r/p^{n-1} \rceil)p^{n-1}+1,p) \neq c(r,(p-\lceil r/p^{n-1} \rceil+1)p^{n-1}+1,p)$, that is, the subperiodic behavior ends at $s=p-\lceil r/p^{n-1} \rceil)p^{n-1}+1$.
\end{enumerate}
\end{theorem}

In fact Glasby, Praeger, and Xia  prove a more general result than the second result of Theorem~GPX that they called \textbf{duality}, but we can view this narrower result above as a reflection inside the period of length $p^n$.  It turns out that there is also a reflection inside the partial subperiod of length $p^{n-1}$ which we describe next.

\begin{theorem}\label{Theorem3}
Suppose that $p^{n-1}<r \leq \frac{p-1}{2} p^{n-1}$.  If $s=b p^{n-1}+d$, then $c(r,s,p)=rev(c(r,b p^{n-1}+(p^{n-1}-d),p))$ whenever $r \leq s=b p^{n-1}+d, b p^{n-1}+p^{n-1}-d \leq (p-\lceil r/p^{n-1} \rceil )p^{n-1}$.
\end{theorem}

Theorems~\ref{Theorem2} and \ref{Theorem3} only have content when the prime $p \geq 5$ since if $p=2$ or $p=3$, $p^{n-1} \not < \frac{p-1}{2} p^{n-1}$.

We illustrate Theorems~\ref{Theorem2} and \ref{Theorem3} with an example.  Let $p=7$ and $r=50$.   Then $n=3$ and $(p-\lceil r/p^{n-1} \rceil )p^{n-1}=(7-2)49=245$.
For an integer $i$ satisfying $1 \leq i \leq 5$, 
\[c(50,49 i+j,7)=
\begin{cases}
(1,1,47,1), & \text{if $j=1$;}\\
(1,j-1,1,48-j,1), & \text{if $2 \leq j \leq 47$;}\\
(1,47,1,1), & \text{if $j=48$;}\\
(1,48,1), & \text{if $j=49$.}
\end{cases}
\]
Thus $c(50,s,7)=c(50,s+49,7)$ for every integer $s \in [50,245]$ and $c(50,49 i+j,7)=c(50,49 i+49-j)$ for every $(i,j) \in \{(k,\ell) \mid 1 \leq k \leq 5, 0 \leq j \leq 49\}\setminus \{(1,0)\}$.  Also $c(50,246,7) =(1,1,47,1)\neq (2,47,1)=c(50,295,7)$.  We see that the most number of parts of the composition $c(50,s,7)$ is smallish at $5$, presumably because $50$ is close to $49=7^2$.  The most number of parts in $c(55,s,7)$ is at least $11$.

In Section~\ref{BarryAl2011}, we  describe the algorithm for calculating $c(r,s,p)$ given by the present author in~\cite{B2015} which we will use to prove our results in Section~\ref{Proofs}.  The paper ends with two observations.

\section{ An Algorithm for Computing $c(r,s,p)$}\label{BarryAl2011}
The present author gave a recursive algorithm in~\cite{B2011_0} for computing $V_m \otimes V_n$ as a sum of indecomposables and rephrased it as an algorithm for computing $c(r,s,p)$ in~\cite{B2015}.  The following notation is used in this algorithm:
if $\mathbf{u}=(u_1,\dots,u_t)$ and $\mathbf{v}=(v_1,\dots,v_f)$ are sequences, then $\mathbf{u}\oplus \mathbf{v}$ denotes $(u_1,\dots,u_t,v_1,\dots,v_f)$, the concatenation of $(u_1,\dots,u_t)$ and $(v_1,\dots,v_f)$.

Assume $0<r \leq s$.  Let $k$ be the unique nonnegative integer $k$ such that $p^k \leq s<p^{k+1}$, so $k=\lfloor \log_p(s) \rfloor$.  Write $s=b p^k+d$ with $1 \leq b <p$ and $0 \leq d <p^k$.  Write $r=a p^k+c$ where $0 \leq a <p$ and $0 \leq c <p^k$. The following six cases are exhaustive and mutually exclusive.

\begin{enumerate}
\item[Case 1] If ($r+s>p^{k+1}$): 

Then $c(r,s,p)=(r+s-p^{k+1}) \oplus c(p^{k+1}-s,p^{k+1}-r,p)$.
\item[Case 2] else if $(r+s  \leq p^{k+1}) \wedge (c+d >p^k)$: \label{Step2}

Here $a+b \leq p-2$.   Then $c(r,s,p)=(c+d-p^k) \oplus c((a+b+1)p^k-s,(a+b+1)p^k-r,p)$
\item[Case 3]  else if $(r+s \leq p^{k+1}) \wedge (1 \leq c+d \leq p^k) \wedge (a>0) $: \label{Step3}

Let $c_1=\min\{c,d\}$, $d_1=\max\{c,d\}$, and define $u=c(c_1,d_1,p)\oplus (d_1-c_1)\oplus rev(c(c_1,d_1,p))$.
Note that if $c_1=0$, $u=(d_1)$.

Then
$c(r,s,p)=u \oplus c((a+b)p^k-s,(a+b)p^k-r,p)$.
\item[Case 4]  else if $(r+s \leq p^{k+1})\wedge (1 \leq c+d \leq p^k) \wedge (a=0) \wedge(d=0)$: \label{Step4}

In this case $r=c$ and $s=b p^k$.  Then  $c(r,s,p)=(c)$.
\item[Case 5] else if $(r+s \leq p^{k+1})\wedge (1 \leq c+d \leq p^k) \wedge (a=0) \wedge(d>0)$: \label{Step5}

In this case $r=c$.   Then
$c(r,s,p)=c(r,b p^k+d,p)=rev(c(r,b p^k-d,p))$.
\item[Case 6] else : \label{Step6} 

Here $(r+s \leq p^{k+1}) \wedge (c=0) \wedge (d=0)$.  So $r=a p^k$ with $a>0$ and $s=b p^k$.  Then
$c(r,s,p)=c(a p^k,b p^k,p)=(p^k) \oplus c((a-1)p^k,(b-1)p^k,p)$.
\end{enumerate}

\section{Proofs}\label{Proofs}

\begin{proof}[Proof of Theorem~\ref{Theorem1}]

Since we know $c(r,s,p)$ has a period of length $p^n$ and because the length of any smaller period would be a divisor of $p^n$, in order to prove that $p^n$ is the length of the smallest period, it suffices to produce $p^{n-1}+1$ distinct $m_1$'s.  

Suppose first that $p^{n-1}<r \leq (p-1)p^{n-1}$, which only happens when $p$ is an odd prime.   Then by Case~1 with $s \in [(p-1)p^{n-1},p^n]$, $c(r,s,p)=(r+s-p^n) \oplus c(p^n-s,p^n-r,p)$.  So $m_1$ ranges from $r-p^{n-1}$ to $r$ giving $p^{n-1}+1$ distinct values of $m_1$.

Consider $r$ such that $r>(p-1)p^{n-1}$.    Then if $r \leq s \leq p^n$, $c(r,s,p)=(r+s-p^n) \oplus c(p^n-s,p^n-r,p)$ by Case~1.  So $m_1$ ranges from $2 r -p^n$ to $r$, giving $r-(2 r-p^n-1)=p^n-r+1$ distinct values of $m_1$.

Now suppose that $s \in [(p-1)p^{n-1}+p^n,r-1+p^n]$.  Then, by Case~2 if $p$ is an odd prime and by Case~1 if $p=2$, $c(r,s,p)=(r+s-2 p^n)  \oplus c(2 p^n-s,2 p^n-r,p)$  Here $m_1$ ranges in value from $r+(p-1)p^{n-1}+p^n-2 p^n$ to $r+r-1+p^n-2 p^n$, that is, from $r-p^{n-1}$ to $2 r-p^n-1$.  We have exhibited values of $m_1$ ranging from $r-p^{n-1}$ to $r$, that is, $p^{n-1}+1$ distinct values of $m_1$.
\end{proof}

\begin{proof}[Proof of Theorem \ref{Theorem2}]
(1) We derive a contradiction by assuming that $r$ is the least integer satisfying $p^{n-1}<r\leq \frac{p-1}{2} p^{n-1}$ for which the result is false.

First we consider the case of $r=a p^{n-1}$ where $1<a \leq \frac{p-1}{2}$. Write $s \in [r,(p-a)p^{n-1}]$ as $b p^{n-1}+d$ where $0 \leq d <p^{n-1}$.    

Consider first the case of $d=0$. Then $r=a p^{n-1}$ and $s =b p^{n-1}$ where $a \leq b \leq p-a$.  By Case~6
\[c(r,s,p)=(p^{n-1})  \oplus c((a-1)p^{n-1},(b-1)p^{n-1},p).\]
By Case~6 again if $b< p-a$,
\[c(r,s+p^{n-1},p)=c(a p^{n-1},(b+1)p^{n-1},p)=(p^{n-1} )\oplus c((a-1)p^{n-1},b p^{n-1},p),\]
whereas if $b=p-a$,
\[c(r,s+p^{n-1},p)=(p^{n-1})\oplus c(p^n-(b+1)p^{n-1},p^n-a p^{n-1},p)=(p^{n-1})\oplus c((a-1)p^{n-1},b p^{n-1},p)\]
by Case~1.  If $a>2$, then $p^{n-1}<(a-1)p^{n-1}<r$ and by the minimality of $r$, $c((a-1)p^{n-1},b p^{n-1},p)=c((a-1)p^{n-1},b p^{n-1},p)$, while if $a=2$, $(a-1)p^{n-1}=p^{n-1}$ and $c((a-1)p^{n-1},b p^{n-1},p)=c((a-1)p^{n-1},b p^{n-1},p)$ by Theorem~\ref{GPX}.  Hence $c(r,s,p)=c(r,s+p^{n-1},p)$ in this case.

Assume $d>0$.  Then Case~3 applies to both $c(r,s,p)$ and $c(r,s+p^{n-1},p)$ with $u=(d)$.  Thus
\begin{align*}
c(r,s,p)&=(d)\oplus c((a+b)p^{n-1}-(b p^{n-1}+d),(a+b)p^{n-1}-a p^{n-1},p)\\
&=(d)\oplus c((a-1)p^{n-1}+p^{n-1}-d,b p^{n-1},p)
\end{align*}
and
\[c(r,s+p^{n-1},p)=(d) \oplus c((a-1)p^{n-1}+p^{n-1}-d,(b+1) p^{n-1},p).\]
If $a>1$, then $p^{n-1}<(a-1)p^{n-1}+p^{n-1}-d<r$ and $c((a-1)p^{n-1}+p^{n-1}-d,b p^{n-1},p)=c((a-1)p^{n-1}+p^{n-1}-d,(b +1)p^{n-1},p)$ by the minimality of $r$, while
if $a=0$, $p^{n-1}-d<p^{n-1}$, and $c(p^{n-1}-d,b p^{n-1},p)=c(p^{n-1}-d,(b +1)p^{n-1},p)$ by Theorem~\ref{GPX}.   Hence $c(r,s,p)=c(r,s+p^{n-1},p)$ in this case.

Now write $r=a p^{n-1}+c$ where where $1 \leq a <(p-1)/2$ and $0<c <p^{n-1}$. Write $s \in [r,(p-\lceil r/p^{n-1} \rceil )p^{n-1}]$ as $b p^{n-1}+d$ where $0 \leq d <p^{n-1}$.  Note that since $c>0$, $p-\lceil r/p^{n-1} \rceil =p-a-1$.  If $c+d>p^{n-1}$, then $d>0$, $b \leq p-a-2$, and Case~2 applies to the computation of $c(r,s,p)$ because $r+s \leq p^n$.  If $c+d \leq p^{n-1}$, then Case~3 applies to the computation of $c(r,s,p)$.

Let's assume that Case~2 applies. Then
\begin{align*}
c(r,s,p)&=c(a p^{n-1}+c,b p^{n-1}+d,p)\\
&=(c+d-p^{n-1}) \oplus c(a p^{n-1}+p^{n-1}-d,b p^{n-1}+p^{n-1}-c,p).
\end{align*}
If $b<p-a-2$, Case~2 also applies to the computation of $c(r,s+p^{n-1},p)$ to give
\begin{align*}
c(r,s+p^{n-1},p)&=c(a p^{n-1}+c,(b+1)p^{n-1}+d,p)\\
&=(c+d-p^{n-1})\oplus c(a p^{n-1}+p^{n-1}-d,(b+1) p^{n-1}+p^{n-1}-c,p).
\end{align*}
On the other hand if $b=p-a-2$, Case~1 applies to the computation of $c(r,s+p^{n-1},p)$ to give
\begin{align*}
c(r,s+p^{n-1},p)&=(r+s-p^n)\oplus c(p^n-((b+1)p^{n-1}-d),p^n-(a p^{n-1}+c),p)\\
&=(c+d-p^{n-1} )\oplus c(a p^{n-1}+p^{n-1}-d, (b+1)p^{n-1}+p^{n-1}-c,p)
\end{align*}
Since $p^{n-1}< a p^{n-1}+p^{n-1}-d<a p^{n-1}+c=r$,
\[c(a p^{n-1}+p^{n-1}-d,b p^{n-1}+p^{n-1}-c,p)=c(a p^{n-1}+p^{n-1}-d,(b +1)p^{n-1}+p^{n-1}-c,p)\]by the minimality of $r$.  Hence 
$c(r,s+p^{n-1},p)=c(r,s,p)$ in this case.

Let's assume that Case~3 applies to the computation of $c(r,s,p)$.  Hence $1 \leq c+d \leq p^{n-1}$ with $c>0$.  We still have $p-\lceil r/p^{n-1} \rceil =p-a-1$.  Then with $c_1=\min\{c,d\}$, $d_1=\max\{c,d\}$, and $u=c(c_1,d_1,p)\oplus (d_1-c_1)\oplus rev(c(c_1,d_1,p))$,
\begin{align*}
c(r,s,p)&=u\oplus c((a+b)p^{n-1}-s,(a+b)p^{n-1}-c,p)\\
&=u \oplus c((a-1)p^{n-1}+p^{n-1}-d,(b-1)p^{n-1}+p^{n-1}-c,p)
\end{align*}
If $d>0$, then $b\leq p-a-2$ and by Case~3 again,
\begin{align*}
c(r,s+p^{n-1},p)&=(u) \oplus c((a+b+1)p^{n-1}-(s+p^{n-1}),(a+b+1)p^{n-1}-r,r)\\
&=(u) \oplus c((a-1)p^{n-1}+p^{n-1}-d,b p^{n-1}+p^{n-1}-c,p).
\end{align*}
If $d=0$, then $b \leq p-a-1$.  If $b\leq p-a-2$, the Case~3 applies again to give $c((r,s+p^{n-1},p)=(u) \oplus c((a-1)p^{n-1}+p^{n-1}-d,b p^{n-1}+p^{n-1}-c,p)$. If $d=0$ and $b=p-a-1$, then $u=(c)$ and Case~1 applies to give
\begin{align*}
c(r,s+p^{n-1},p)&=(c)\oplus c(p^n-(s+p^{n-1}),p^n-r,p)\\
&=(c)\oplus c(p^n-(p-a)p^{n-1},p^n-(a p^n+c),p)\\
&=u \oplus c(a p^{n-1},b p^{n-1}+p^{n-1}-c,p)\\
&=u \oplus c((a-1)p^{n-1}+p^{n-1}-d,b p^{n-1}+p^{n-1}-c,p)
\end{align*}

If $a>1$, then $p^{n-1} <a p^{n-1}+p^{n-1}-d<r$ and therefore $c((a-1)p^{n-1}+p^{n-1}-d,(b-1)p^{n-1}+p^{n-1}-c,p)=c((a-1)p^{n-1}+p^{n-1}-d,b p^{n-1}+p^{n-1}-c,p)$by the minimality of $r$.  If $a=0$, $p^{n-1}-d \leq p^{n-1}$, and $c(p^{n-1}-d,(b-1) p^{n-1}+p^{n-1}-c,p)=c(p^{n-1}-d,b p^{n-1}+p^{n-1}-c,p)$ by Theorem~\ref{GPX}.  Hence 
$c(r,s+p^{n-1},p)=c(r,s,p)$ in this case.

(2) Now we show that $c(r,(p-\lceil r/p^{n-1} \rceil)p^{n-1}+1,p) \neq c(r,(p-\lceil r/p^{n-1} \rceil+1)p^{n-1}+1,p)$.  First consider the case of $r=a p^{n-1}+c$ where $0<c<p^{n-1}$.  Then $(p-\lceil r/p^{n-1} \rceil)p^{n-1}+1=(p-a-1)p^{n-1}+1$.  Case~3 applies here because $c+1 \leq p^{n-1}$.  Since $\lambda(1,c,p)=(c)$, $c(1,c,p)=(1)$and $u=(1,c-1,1)$.  Thus the first component of $c(r,(p-\lceil r/p^{n-1} \rceil)p^{n-1}+1,p)$ is $1$.  Case~1 applies in the computation of $c(r,(p-\lceil r/p^{n-1} \rceil+1)p^{n-1}+1,p)=c(a p^{n-1}+c,(p-a)p^{n-1}+1,p)$ to give a first component of $c+1$.  Now consider the case of $r=a p^{n-1}$ where $2 \leq a \leq (p-1)/2$.  Case~1 applies in the computation of $c(a p^{n-1},(p-\lceil r/p^{n-1} \rceil)p^{n-1}+1,p) =c(a p^{n-1},(p-a)p^{n-1}+1,p)$ to give a first component of $1$, while Case~1 also applies in the computation of $c(r,(p-\lceil r/p^{n-1} \rceil+1)p^{n-1}+1,p)=c(a p^{n-1}, (p-a+1) p^{n-1}+1,p)$ to give a first component of $p^{n-1}+1$.
\end{proof}

\begin{proof}[Proof of Theorem~\ref{Theorem3}]
Assume that $0 \leq d < p^{n-1}$ and that $r \leq b p^{n-1}+d, b p^{n-1}+p^{n-1}-d \leq (p-\lceil r/p^{n-1} \rceil )p^{n-1}$.  We will assume that $d>0$ since the result follows from Theorem~\ref{Theorem2} when $d=0$.  We will prove the following three equalities:
\begin{align*}
c(r,b p^{n-1}+(p^{n-1}-d),p)&=c(r,p^n+b p^{n-1}+(p^{n-1}-d),p)\\
&=rev(c(r,p^n-(b p^{n-1}+p^{n-1}-d),p))\\
&=rev(c(a p^{n-1},b p^{n-1}+d,p)).
\end{align*}

The first equality, $c(r,b p^{n-1}+(p^{n-1}-d),p)=c(r,p^n+b p^{n-1}+(p^{n-1}-d),p)$, follows by  Theorem~\ref{GPX}.  Now write $r=a p^{n-1}+c$ where $0 \leq c < p^{n-1}$.  In order to apply Case~5 we need $c +b p^{n-1}+p^{n-1}-d \leq p^n$.  If $c=0$, 
\[c +b p^{n-1}+p^{n-1}-d=b p^{n-1}+p^{n-1}-d\leq (p-\lceil r/p^{n-1} \rceil )p^{n-1}=(p-a)p^{n-1}<p^n.\] 
On the other hand if $c>0$, 
\[c +b p^{n-1}+p^{n-1}-d \leq c+(p-\lceil r/p^{n-1} \rceil )p^{n-1} =c+(p-a-1)p^{n-1}<(p-a)p^{n-1}<p^n.
\]
By Case~5, $c(r,p^n+b p^{n-1}+(p^{n-1}-d),p)=rev(c(r,p^n-(b p^{n-1}+p^{n-1}-d),p))$.  

Our final step is to apply Theorem~\ref{Theorem2} to get $c(r,(p-b-1)p^{n-1}+d,p)=c(r,b p^{n-1}+d,p)$,  but in order to do this, we must verify that
\[r \leq (p-b-1)p^{n-1}+d \leq (p-\lceil r/p^{n-1} \rceil +1)p^{n-1}.
\]
The presence of the term $p-\lceil r/p^{n-1} \rceil +1$ rather that $p-\lceil r/p^{n-1} \rceil$ in the last displayed equation is justified by the fact that in Theorem~\ref{Theorem2} along as $r \leq s \leq (p-\lceil r/p^{n-1} \rceil)p^{n-1}$, $c(r,s,p)=c(r,s+p^{n-1},p)$.

If $r =a p^{n-1}+c$ with $c>0$, then $p-\lceil r/p^{n-1} \rceil=p-a-1$.  Because $b p^{n-1}+d \leq (p-a-1) p^{n-1}$ and $d>0$, $b \leq p-a-2$, so $p-b-1 \geq a+1$ and $r<(p-b-1)p^{n-1}+d $.  Since $a \leq b$, 
\[(p-b-1)p^{n-1}+d \leq (p-a-1)p^{n-1}+d <(p-a)p^{n-1}=(p-\lceil r/p^{n-1} \rceil +1)p^{n-1}.
\]

If $r=a p^{n-1}$ with $2 \leq a \leq(p-1)/2$, then $p-\lceil r/p^{n-1} \rceil=p-a$. Because $b p^{n-1}+d \leq (p-a)p^{n-1}$ and $d>0$, then $b \leq p-a-1$, so $p-b-1 \geq a$ implying $r<(p-b-1)p^{n-1}+d$.  Since $b \geq a$,
\[(p-b-1)p^{n-1}+d <(p-a)p^{n-1}=(p-\lceil r/p^{n-1} \rceil)p^{n-1}<(p-\lceil r/p^{n-1} \rceil +1)p^{n-1}.\]
By Theorem~\ref{Theorem2}, $c(a p^{n-1},(p-b-1)p^{n-1}+d,p)=c(a p^{n-1},b p^{n-1}+d,p)$.  These steps together prove that $c(r,b p^{n-1}+(p^{n-1}-d),p)=rev(c(r,b p^{n-1}+d,p))$.
\end{proof}

\section{Two Concluding Observations}

(1) One can ask what happens when $p^{n-1}<r \leq \frac{p-1}{2}$ and $(p-\lceil r/p^{n-1} \rceil)p^{n-1}<s<p^n$.  We claim that  $c(r,s,p)=rev(c(r,2 p^n-s,p))$.  Since $r+(p^n-s)<p^n$, $c(r,p^n+p^n-s,p)=rev(c(r,p^n-(p^n-s),p))=rev(c(r,s,p))$ by Case~5. So what happens is a reflection around  $s=p^n$.

(2) By Theorem~\ref{GPX}, $c(r,s+p^n,p)=c(r,s,p)$, in particular $c(r,r+p^n,p)=c(r,r,p)$.  But what can we say about $c(r,r_1+p^n,p)$ where $r_1<r$ but close to $r$?  We give a two-part answer here whose proof is very similar to the proof of Theorem~\ref{Theorem3}.
\begin{enumerate}
\item[(i)] If $p^{n-1}<r =ap^{n-1}+c \leq \frac{p-1}{2} p^{n-1}$ where $0<c<p^{n-1}$, then $c(r,a p^{n-1}+c_1+p^n,p)=c(r,(a+1)p^{n-1}+c_1,p)$ for every $c_1 \in [0,c)$.
\item[(ii)] If $r=a p^{n-1}$ where $2 \leq a \leq \frac{p-1}{2}$, then $c(a p^{n-1},(a-1)p^{n-1}+d+p^n,p)=c(a p^{n-1},a p^{n-1}+d,p)$ for every $d\in [0,p^{n-1})$.
\end{enumerate}
So the subperiod kicks in a little earlier when we add $p^n$ to $s$.

\end{document}